\documentclass[12pt,a4paper]{amsart}
\usepackage{amssymb,amsmath}
\usepackage[arrow, matrix, curve]{xy}
\usepackage{hyperref}

\numberwithin{equation}{section}
\newtheorem{theorem}{Theorem}[section]
\newtheorem{question}[theorem]{Question}
\newtheorem{lemma}[theorem]{Lemma}

\newtheorem{corollary}[theorem]{Corollary}

\theoremstyle{definition}
\newtheorem{definition}[theorem]{Definition}

\newtheorem{remark}[theorem]{Remark}

\newcommand\Supp{\operatorname{Supp}}
\newcommand\As{\operatorname{As}}
\newcommand{\Ass}[0]{\operatorname{Ass}}

\newcommand\Tor{\operatorname{Tor}}
\newcommand\Hom{\operatorname{Hom}}

\newcommand\Ext{\operatorname{Ext}}
\newcommand\Rad{\operatorname{Rad}}
\newcommand\Ker{\operatorname{Ker}}

\newcommand\depth{\operatorname{depth}}

\newcommand\grade{\operatorname{grade}}

\newcommand\Spec{\operatorname{Spec}}
\newcommand{\gam}{\Gamma_{I}}
\newcommand{\Rgam}{{\rm R} \Gamma_{\mathfrak m}}

\title{Notes on local cohomology and duality}

\author{Michael Hellus and Peter Schenzel}

\address{Fakult\"at f\"ur Mathematik, Universit\"at Regensburg, D --- 93040 Regensburg, Germany}

\email{michael.hellus@mathematik.uni-regensburg.de}

\address{Martin-Luther-Universit\"at Halle-Wittenberg,
Institut f\"ur Informatik, D --- 06 099 Halle (Saale),
Germany}

\email{peter.schenzel@informatik.uni-halle.de}



\begin{document}

\begin{abstract}

We provide a formula (see Theorem \ref{MDE}) for the Matlis dual of the injective hull of $R/\mathfrak{p}$ where $\mathfrak p$ is a one dimensional prime ideal in a local complete Gorenstein domain $(R,\mathfrak{m})$. This  is related to results of Enochs and Xu (see \cite{E} and \cite{EX}). We prove a certain 'dual' version of the Hartshorne-Lichtenbaum vanishing (see Theorem \ref{HLVTL}). There is a
generalization of local duality to cohomologically complete intersection ideals $I$ in the sense that for $I=\mathfrak{m}$ we get back the classical Local Duality Theorem. We determine the exact class of modules to which a characterization of cohomologically complete intersection from \cite{hellus_schenzel} generalizes naturally (see Theorem \ref{th_gvococciwrtM}).

\end{abstract}
\maketitle
In this paper we prove a Matlis dual version of Hartshorne-Lichtenbaum Vanishing Theorem and generalize the Local Duality Theorem.    The latter generalization is done for ideals which are cohomologically complete intersections, a notion which was introduced and studied in \cite{hellus_schenzel}.  The generalization is such that local duality becomes the special case when the ideal $I$ is the maximal ideal $\mathfrak{m}$ of the given local ring ($R,\mathfrak{m}$). We often use formal local cohomology, a notion which was introduced and studied by the second author in \cite{S2}. Formal local cohomology is related to Matlis duals of local cohomology modules (see \cite[Sect.~7.1~and~7.2]{hellushabil} and Corollary \ref{zz}).

We start in Section 1 with the study of the Matlis duals of local cohomology modules $H^{n-1}_I(R)$, where $n=\dim R$. The latter is also the formal local cohomology module $\varprojlim H^1_{\mathfrak{m}}(R/I^{\alpha})$ provided $R$ is a  Gorenstein ring. We describe  this module as the cokernel of a certain canonical map. As a consequence we derive a formula (see Theorem \ref{MDE}) for the Matlis dual of $E_R(R/\mathfrak{p})$, where $\mathfrak{p}\in \Spec R$ is a $1$-dimensional prime ideal. In some sense this is related to results by Enochs and Xu (see \cite{E} and \cite{EX}).

In Section 3 we generalize the Local Duality (see Theorem \ref{xx}). The canonical module in the classical version is replaced by the dual of $H^c_I(R)$ where $I$ is a cohomologically complete intersection ideal of grade $c$ (the case $I=\mathfrak{m}$ specializes to the classical local duality). See also \cite[Theorem 6.4.1]{hellushabil}. It is a little bit surprising that the $d$-th formal local cohomology occurs as the duality module for the duality of cohomologically complete intersections in a Gorenstein ring (see Corollary \ref{zz}).

In Theorem \ref{th_gvococciwrtM} we generalize the main result \cite[Theorem 3.2]{hellus_schenzel}. This provides  a
characterization of the property of 'cohomologically complete intersection' given for ideals to  finitely generated modules.
Finally, in Section 5 we fill a gap in our proof of \cite[Lemma 1.2]{hellus_schenzel}. To this end we use a result on inverse limits
as it was shown by the second author (see \cite{S3}). Some of the results of Section 4 are obtained independently by W. Mahmood
(see \cite{M}).

\section{On formal local cohomology}
Let $(R,\mathfrak m)$ be a local ring, let $I \subset R$
be an ideal. In the following let $\hat R^I$ denote the $I$-adic completion
of $R.$ Let $0 = \cap_{i=1}^r \mathfrak q_i$ denote a minimal primary
decomposition of the zero ideal. Then we denote by $u(I)$ the intersections of those
$\mathfrak q_i, i = 1,\ldots,r,$ such that $\dim R/(\mathfrak p_i + I)  > 0,$ where
$\Rad \mathfrak{q}_i = \mathfrak{p}_i, i = 1,\ldots,r.$

For the definition and basic properties
of $\varprojlim H^i_{\mathfrak{m}}(R/I^{\alpha}),$ the so-called formal local cohomology,
we refer to \cite{S2}. We denote the functor of global transform by
$T(\cdot) = \varinjlim \Hom_R(\mathfrak m^{\alpha}, \cdot)$, in order to distinguish it from Matlis duality
\[
D(M)=\Hom_R(M,E_R(R/\mathfrak{m})),
\]
where $E_R(R/\mathfrak{m})$ is a fixed $R$-injective hull of $k:=R/\mathfrak{m}$.

\begin{lemma}

\label{lemma_vorber}Let $I \subset R$ denote an arbitrary ideal. Then there is a short exact
sequence
\[
0 \to \hat  R^I/u(I \hat R^I)  \to \varprojlim T(R/I^{\alpha}) \to \varprojlim H^1_{\mathfrak m}(R/I^{\alpha})
\to 0.
\]

\end{lemma}

\begin{proof} For each $\alpha \in \mathbb N$ there is the following canonical
exact sequence
\[
0 \to H^0_{\mathfrak m}(R/I^{\alpha}) \to R/I^{\alpha} \to T(R/I^{\alpha}) \to H^1_{\mathfrak{m}}(R/I^{\alpha})
\to 0.
\]
It splits up into two short exact sequences
\begin{gather*}
0 \to H^0_{\mathfrak m}(R/I^{\alpha}) \to R/I^{\alpha} \to R/I^{\alpha} : \langle \mathfrak{m}\rangle \to 0  \text{ and }\\
0 \to  R/I^{\alpha} : \langle \mathfrak{m}\rangle \to  T(R/I^{\alpha}) \to H^1_{\mathfrak{m}}(R/I^{\alpha})
\to 0.
\end{gather*}
Now the inverse systems at the left side of both of the exact sequences satisfy the Mittag-Leffler
condition. That is, by passing to the inverse limits it provides two short exact sequences. Putting them
together there is an exact sequence
\[
0 \to \varprojlim H^0_{\mathfrak m}(R/I^{\alpha})  \to \hat{R}^I \to
\varprojlim T(R/I^{\alpha}) \to \varprojlim H^1_{\mathfrak m}(R/I^{\alpha})
\to 0.
\]
Now it follows that $\varprojlim H^0_{\mathfrak m}(R/I^{\alpha})  = u(I\hat{R}^I),$ see \cite[Lemma 4.1]{S2}.
This finally proves the statement.
\end{proof}

Of a particular interest in the above Corollary is the case where $I \subset R$ is an ideal such that
$\dim R/I = 1.$

\begin{corollary}

Suppose that $\dim R/I = 1.$ Then there is a short
exact sequence
\[
0 \to \hat R^I/u(I \hat R^I)  \to \oplus_{i =1}^s \widehat{R_{\mathfrak p_i}}
\to \varprojlim H^1_{\mathfrak m}(R/I^{\alpha}) \to 0,
\]
where $\mathfrak p_i, i = 1,\ldots,s,$ denote those prime ideals $\mathfrak p$ of $\Ass R/I$ such that $\dim R/\mathfrak p = 1.$

\end{corollary}

\begin{proof} Since $\dim R/I = 1$ there is an element $x \in \mathfrak{m}$ that is a parameter
for $R/I^{\alpha}$ for all $\alpha \in \mathbb{N}.$ Then there is an isomorphism
$T(R/I^{\alpha}) \simeq R_x/I^{\alpha}R_x$ for all $\alpha \in \mathbb{N}.$

Now let $S = \cap_{i=1}^s (R \setminus \mathfrak{p}_i).$ Since $x \in S$ there is a natural isomorphism (by the local-global-principle)
\[
R_x/I^{\alpha}R_x \simeq R_S/I^{\alpha}R_S  \text{ for all } \alpha \in \mathbb{N}.
\]
Then $R_S$ is a semi local ring. The Chinese Remainder Theorem provides isomorphisms
\[
R_S/I^{\alpha}R_S \simeq \oplus_{i=1}^s R_{\mathfrak{p}_i}/I^{\alpha}R_{\mathfrak{p}_i}
\text{ for all } \alpha \in \mathbb{N}.
\]
Now $\Rad I R_{\mathfrak{p}_i}= \mathfrak{p_i}R_{\mathfrak{p}_i}, i = 1,\ldots,s.$ So by passing
to the inverse limit we get the isomorphism
\[
\varprojlim T(R/I^{\alpha}) \simeq \oplus_{i=1}^s \widehat{R_{\mathfrak{p}_i}}.
\]
Therefore \ref{lemma_vorber} finishes the proof of the statement.
\end{proof}

\begin{remark}

In the case that $R/I$ is one dimensional it follows that
\[
H^1_{\mathfrak m}(R/I^{\alpha}) \simeq H^1_x(R/I^{\alpha}) \simeq
H^1_x(R) \otimes R/I^{\alpha} \simeq (R_x/R) \otimes R/I^{\alpha}
\]
for all $\alpha \in \mathbb N,$ where $x \in \mathfrak{m}$ denotes a parameter
of $R/I.$ This finally implies that
\[
\varprojlim H^1_{\mathfrak m}(R/I^{\alpha}) \simeq \widehat{R_x/R}^I.
\]

\end{remark}

\begin{corollary}

\label{cor_DHnmE}Suppose that $I \subset R$ is a one dimensional ideal in a local Gorenstein
ring $(R,\mathfrak{m})$ with $n = \dim R.$  Then there is a short exact sequence
\[
0 \to \hat R^I/u(I \hat R^I)  \to \oplus_{i =1}^s \widehat{R_{\mathfrak p_i}}
\to \Hom_R(H^{n-1}_I(R), E) \to 0,
\]
where $\mathfrak p_i, i = 1,\ldots,s,$ denote those prime ideals $\mathfrak p$ of $\Ass R/I$ such that $\dim R/\mathfrak p = 1.$

\end{corollary}

\begin{proof} This is clear because of $\varprojlim H^1_{\mathfrak m}(R/I^{\alpha})  \simeq
\Hom_R(H^{n-1}_I(R), E) $ as it is a consequence of the Local Duality Theorem for Gorenstein
rings (the Hom-functor in the first place transforms a direct limit into an inverse limit).
\end{proof}

In particular the Matlis dual of $H^{n-1}_I(R)$ is exactly the cokernel of the canonical map $\hat R^I  \to \oplus_{i =1}^s \widehat{R_{\mathfrak p_i}}$ This generalizes \cite[Lemma 3.2.1]{hellushabil} (see also \cite[Lemma 1.5]{hellus_stueckrad}).

If we assume in addition that $I=\mathfrak{p}$ is a one dimensional prime ideal and that $R$ is a complete domain, then by \cite[Theorem 3.2]{hellus_schenzel} the fact $H^n_\mathfrak{p}(R)=0$ (as follows by the Hartshorne-Lichtenbaum Vanishing Theorem) is equivalent to: The minimal injective resolution of $H^{n-1}_\mathfrak{p}(R)$ looks as follows:
\[ 0\to H^{n-1}_\mathfrak{p}(R)\to E_R(R/\mathfrak p)\to E_R(R/\mathfrak m)\to 0.\]
On the other hand we have (see Corollary \ref{cor_DHnmE}) a short exact sequence
\[ 0\to R\to \widehat{R_\mathfrak p}\to D(H^{n-1}_\mathfrak p(R))\to 0.\]
Note that the natural map $R\to \widehat{R_\mathfrak p}$ is injective because $R$ is a complete domain.
Therefore, $u(\mathfrak{p})=u(\mathfrak{p}\hat R^\mathfrak{p})= 0$. The comparison of the two exact sequences has
the following consequence:

Applying the functor $D$ to the first exact sequence it induces a natural homomorphism
\[
 R=D(E_R(R/\mathfrak m))\to D(E_R(R/\mathfrak p)).
\]
Since the latter is an $R_\mathfrak p$-module, we get a map
$R_\mathfrak p\to D(E_R(R/\mathfrak p))$ and therefore a family of homomorphisms
\[
R_\mathfrak p/\mathfrak{p}^{\alpha}R_{\mathfrak{p}} \to  D(E_R(R/\mathfrak p))/\mathfrak{p}^{\alpha}D(E_R(R/\mathfrak p))
\]
for any integer $\alpha \in \mathbb{N}$. But now we have the isomorphisms
\begin{gather*}
D(E_R(R/\mathfrak p))=\Hom_R(\varinjlim_{\alpha}\Hom_R(R/\mathfrak p^{\alpha},E_R(R/\mathfrak p)),E_R(R/\mathfrak m))= \\
=\varprojlim_{\alpha}D(E_R(R/\mathfrak p))/\mathfrak p^{\alpha}D(E_R(R/\mathfrak p)).
\end{gather*}
Therefore the above inverse systems induce a homomorphism $f: \widehat{R_\mathfrak p}\to D(E_R(R/\mathfrak p))$.
Clearly the natural homomorphism $R=D(E_R(R/\mathfrak m))\to D(E_R(R/\mathfrak p))$ factors through
$f$.  So the above two short exact sequences induce a commutative diagram
\[
\begin{xy}
\xymatrix{
0\ar[r]&R\ar[r]\ar@{=}[d]&\widehat{R_\mathfrak p}\ar[r]\ar[d]&D(H^{n-1}_\mathfrak p(R))\ar[d]\ar[r]&0\\
0\ar[r]&D(E_R(R/\mathfrak m))=R\ar[r]&D(E_R(R/\mathfrak p))\ar[r]&D(H^{n-1}_\mathfrak p(R))\ar[r]&0}
\end{xy}
\]
All maps in this commutative diagram are canonical and it is easy to see that the vertical homomorphism on the right is the identity. Therefore $f$ is an isomorphism too.
We conclude with the following result:

\begin{theorem}

\label{MDE}Let $\mathfrak p$ be a prime ideal of height $n-1$ in an $n$-dimensional local, complete Gorenstein domain $(R,\mathfrak m)$. Then the Matlis dual of $E_R(R/\mathfrak p)$ is $\widehat{R_\mathfrak p}$.

\end{theorem}

This is related to results from Enochs and Xu: $D(E_R(R/\mathfrak p))$ is flat and cotorsion by \cite[Lemma 2.3]{E} (see also \cite[Theorem 1.5]{I}), therefore (see \cite[Theorem]{E}), it is isomorphic to a direct product of modules $T_\mathfrak q$ (over $\mathfrak q\in \Spec R$) where each $T_\mathfrak q$ is the completion of a free module over $R_\mathfrak q$. It was also proved in \cite{E} that in this direct product the ranks of these free modules are uniquely determined. By \cite[Theorem 2.2]{EX} these ranks are
\[ \pi_0(\mathfrak q,D(E_R(R/\mathfrak p)))=\dim _{k(\mathfrak q)}k(\mathfrak q)\otimes _{R_\mathfrak q}\Hom_R(R_\mathfrak q,D(E_R(R/\mathfrak p)))\]
(all higher $\pi_i$ for $i>0$ vanish since $D(E_R(R/\mathfrak q))$ is flat: Its minimal flat resolution is trivial). For each $\mathfrak q$ different from $\mathfrak p$ the latter rank is zero: In case $\mathfrak p\neq \mathfrak m$ this follows from
\[ \Hom_R(R_\mathfrak q,D(E_R(R/\mathfrak p)))=D(R_\mathfrak q\otimes _RE_R(R/\mathfrak p))=0,\]
and in case $\mathfrak q=\mathfrak m$ we have
\[ (R/\mathfrak m)\otimes_RD(E_R(R/\mathfrak p))=D(\Hom_R(R/\mathfrak m,E_R(R/\mathfrak p)))=0.\]
Therefore, the use of those results from \cite{E}, \cite{EX} leads us to
\[ D(E_R(R/\mathfrak p))=T_\mathfrak p,\]
where $T_\mathfrak p$ is the completion of a free $R_\mathfrak p$-module.

Our Theorem \ref{MDE} gives the more precise information that the rank of this free module is exactly $1$, i.~e. $T_\mathfrak p\cong \widehat{R_\mathfrak p}$.

\begin{question}

Is it possible to deduce the fact that this rank is $1$ directly from \cite[Theorem 2.2]{EX}, i.~e., without using our theorem \ref{MDE}?

\end{question}

\section{A Remark on the Hartshorne-Lichtenbaum Vanishing Theorem}
\label{sect_rem_HLVT}

In this Section there is a comment on the Hartshorne-Lichtenbaum Vanishing Theorem in view to
the previous investigations. Let $I$ denote an ideal in a local Noetherian ring $(R,\mathfrak{m})$
with $\dim R = n.$ As above let  $\hat R^I$ denote the $I$-adic completion
of $R.$ Let $0 = \cap_{i=1}^r \mathfrak q_i$ denote a minimal primary
decomposition of the zero ideal. Then we denote by $v(I)$ the intersection of those
$\mathfrak q_i, i = 1,\ldots,r,$ such that $\dim R/(\mathfrak q_i + I)  > 0$ and $\dim R/\mathfrak{q}_i = n$.
Recall that $v(I) = u(I)$ (for the ideal $u(I)$ as introduced at the beginning of Section 1) if $R$ is
equi-dimensional.

The following result provides a variant of the Hartshorne Lichtenbaum Vanishing Theorem.

\begin{theorem} \label{HLVT} (\cite[Theorem 2.20]{S3})
Let $I \subset R$ denote an ideal and $n = \dim R$. Then
\[
H^n_I(R) \cong \Hom_R(v(I\hat{R}), E_R(R/\mathfrak{m})).
\]
That is $H^n_I(R)$ is an Artinian $R$-module and $H^n_I(R) = 0$ if and only if
$\dim \hat{R}/(I\hat{R} + \mathfrak{p} )> 0$ for all $\mathfrak{p} \in \Ass \hat{R}$ with
$\dim \hat{R}/\mathfrak{p} = n$.
\end{theorem}

For an ideal $I$ of a Noetherian ring $R$ let $\As I$ denote the ultimate constant (see \cite{B}) value of
$\Ass R/I^{\alpha}$ for $\alpha \gg 0.$ We define the multiplicatively closed set $S = \cap_{\mathfrak{p} \in
\As I \setminus \{\mathfrak{m}\}} R \setminus \mathfrak{p}.$
Then there is an exact sequence
\[
0 \to I^{\alpha} :\langle\mathfrak{m} \rangle /I^{\alpha} \to R/I^{\alpha} \to R_S/I^{\alpha}R_S \; \text{ for all }
\alpha \gg 0.
\]
Since the modules on the left are of finite length the corresponding inverse system satisfies
the Mittag-Leffler condition. By passing to the inverse limit it induces an exact sequence
\[
0 \to u(I \hat R^I) \to \hat{R}^I \to \widehat{R_S}^I
\]
(see \cite[Lemma 4.1]{S2}).
Now let $R$ denote a complete equidimensional local ring. Then the natural homomorphism
$R \to \widehat{R_S}^I$ is injective if and only if $H^n_I(R) = 0.$ This follows since $v(I) = u(I) = 0$
if and only if $H^n_I(R) = 0$ under the additional assumption on $R.$

If in addition  $\dim R/I = 1$ we have as above that $R_S/I^{\alpha}R_S \simeq \oplus_{i=1}^s
R_{\mathfrak{p}_i}/I^{\alpha}R_{\mathfrak{p}_i}.$ Therefore, if $I$ is a one dimensional
ideal in an equidimensional complete local ring $(R,\mathfrak{m}).$  Then the natural homomorphism
$R \to  \oplus_{i =1}^s \widehat{R_{\mathfrak p_i}} $ is injective if and only if $H^n_I(R) = 0$ (the $\mathfrak{p}_i$ are defined as above). In case $R$ is in addition a domain then our map $R \to  \oplus_{i =1}^s \widehat{R_{\mathfrak p_i}} $ is clearly injective and, therefore, $H^n_I(R)=0$.

In the following we shall continue with this series of ideas in the case of $(R,\mathfrak{m})$
a Gorenstein ring.  To this end we put $V(I)_1 = \{\mathfrak{p} \in V(I) | \dim R/\mathfrak{p} = 1\}$.

\begin{theorem} \label{HLVTL} Let $(R,\mathfrak{m})$ denote a Gorenstein ring with $n = \dim R$. Let $I \subset R$
denote an ideal. Then $D(H^n_I(R))$ is isomorphic to the kernel of the natural
map $\hat{R} \to \prod_{\mathfrak{p} \in V(I)_1} \widehat{R_{\mathfrak{p}}}$. In particular, this homomorphism
is injective if and only if $H^n_I(R) = 0$.
\end{theorem}

\begin{proof} By applying the section functor $\Gamma_I(\cdot)$ to the minimal injective resolution of the
Gorenstein ring $R$ it provides an exact sequence
\[
\oplus_{\mathfrak{p} \in V(I)_1} E_R(R/\mathfrak{p}) \to E_R(R/\mathfrak{m}) \to H^n_I(R) \to 0.
\]
Now apply the Matlis duality functor $D(\cdot)$ to the sequence. It provides the exact sequence
\[
0 \to D(H^n_I(R)) \to \hat{R} \to \prod_{\mathfrak{p} \in V(I)_1} D(E_R(R/\mathfrak{p})).
\]
By virtue of Theorem \ref{MDE} it follows that $D(E_R(R/\mathfrak{p})) \cong \widehat{R_{\mathfrak{p}}}$.
Therefore $D(H^n_I(R))$ is isomorphic to the kernel of the natural map
$\hat{R} \to \prod_{\mathfrak{p} \in V(I)_1} \widehat{R_{\mathfrak{p}}}$. Matlis duality provides the claim.
\end{proof}

If in addition $\dim R/I = 1$, then $V(I)_1$ is a finite set. Therefore the direct product in Theorem
\ref{HLVTL} is in fact a direct sum. Hence, the result in Theorem \ref{HLVTL} is a generalization
of the considerations above for the case $\dim R/I = 1$.

Moreover, in a certain sense Theorem \ref{HLVTL} is a dual version to \cite[Proposition]{CS} 
shown by Call and Sharp. 

\section{On a duality for cohomologically complete intersections}
\label{sect_on_dual_cci}As above let $(R,\mathfrak{m})$ denote a local Noetherian ring. An ideal $I \subset R$ is called
a cohomologically complete intersection whenever $H^i_I(R) = 0$ for all $i \not= c$ for some $c$ (see \cite{hellus_schenzel} for
the definition and a characterization). If $I$ is a cohomologically complete interssection, then
in the paper of Zargar and Zakeri (see \cite{ZZ}) the ring $R$ is called Cohen-Macaulay with respect
to $I$.

The main aim of the present section is to prove a generalized local duality for  a cohomologically
complete intersection $I$. A corresponding result was already obtained by W. Mahmood (see \cite{M}) resp. by the second author in \cite[Theorem 6.4.1]{hellushabil} by different means.

\begin{theorem} \label{xx} Let $I \subset R$ denote a cohomologically complete intersection with $c = \grade I$.
Let $X$ denote an arbitrary $R$-module. Then there are the following functorial isomorphisms
\begin{itemize}
	\item[(a)] $\Tor_{c-i}^R(X,H^c_I(R)) \cong H^i_I(X)$ and
	\item[(b)] $\Ext_R^{c-i}(X, \Hom_R(H^c_I(R),E_R(k))) \cong \Hom_R(H^i_I(X),E_R(k))$
\end{itemize}
for all $i \in \mathbb{Z}$.
\end{theorem}

\begin{proof} First of all choose $\underline{x} = x_1,\ldots,x_r$ a system of elements of $R$
such that $\Rad \underline{x}R = \Rad I$. Then we consider the \v{C}ech complex
$\Check{C}_{\underline x}$. This is a bounded complex of flat $R$-modules with
 $H^i(\Check{C}_{\underline x}) = 0$
for all $i \not= c$ and $H^c(\Check{C}_{\underline x}) \cong H^c_I(R)$. Moreover, $H^i(X \otimes_R \Check{C}_{\underline x}) \cong H^i_I(X)$ (see e.g. \cite{S}). In order to compute the
cohomology $\Tor_i(\Check{C}_{\underline x}, X)$ there is the following spectral sequence
\[
E^{i,j}_2 = \Tor_{-i}^R(H^j_I(R),X) \Rightarrow E_{\infty}^{i+j} = \Tor_{-i-j}^R(\Check{C}_{\underline x},X):
\]
Since $I$ is a cohomologically complete intersection we get a degeneration to the following
isomorphisms
\[
\Tor_{c-i}^R(H^c_I(R),X) \cong \Tor_i^R(\Check{C}_{\underline x},X) \cong H^i_I(X)
\]
for all $i \in \mathbb{Z}$. This proves the isomorphisms of the statement in (a).

For the proof of (b) note that
\[
\Hom_R(\Tor_{c-i}^R(H^c_I(R),X), E_R(k)) \cong \Ext_R^{c-i}(X,\Hom_R(H^c_I(R),E_R(k)))
\]
as follows by adjunction since $E_R(k)$ is an injective $R$-module.
\end{proof}

The Matlis dual $\Hom_R(H^c_I(R),E_R(k))=D(H^c_I(R))$ plays a
central r\^ole in the above generalized duality. It allows to express the Matlis dual of
$H^i_I(X)$ in terms of an Ext module.

\begin{definition} Let $I \subset R$ denote a cohomologically complete intersection
with $c = \grade I$. Then we call $D_I(R) = \Hom_R(H^c_I(R), E_R(k))=D(H^c_I(R))$ the duality module
of $I$.
\end{definition}

In general the structure of $D_I(R)$ is difficult to determine. In the following we want to
discuss a few particular cases of cohomologically complete
intersections and their duality module. To this end let $K(R)$ denote the canonical
module of $R$, provided it exists.

\begin{corollary} \label{yy} Let $(R, \mathfrak{m})$ denote a local ring such that $\mathfrak{m}$
is a cohomologically complete intersection. There are natural isomorphisms
\[
H^{n-i}_{\mathfrak{m}}(M) \cong \Hom_R(\Ext_R^i(M,K(\hat{R}),E_R(k))), n = \dim R,
\]
for a finitely generated $R$-module $M$ and any $i \in \mathbb{Z}$. Note that $R$ is Cohen-Macaulay ring and $D_{\mathfrak{m}}(R)
\cong K(\hat{R})$.
\end{corollary}

\begin{proof} In case $\mathfrak{m}$ is a cohomologically complete intersection, then
$\depth R = \grade \mathfrak{m} = \dim R$ since $H^c_{\mathfrak{m}}(R)$ is the onliest
non-vanishing local cohomology module. This follows by the non-vanishing of $H^i_{\mathfrak{m}}(R)$
for $i = \depth R$ and $i = \dim R$ (see e.g. \cite{S3}). Therefore $R$ is a Cohen-Macaulay ring.

Moreover, $\Hom_R(H^d_{\mathfrak{m}}(R),E_R(k)) \cong \Hom_{\hat{R}}(H^d_{\hat{\mathfrak{m}}}(\hat{R}), (E_{\hat{R}}(k)))$
since $E_R(k)$ is an Artinian $R$-module. Then $\hat{R}$ admits a canonical module and
$K(\hat{R}) \cong \Hom_{\hat{R}}(H^d_{\hat{\mathfrak{m}}}(\hat{R}), (E_{\hat{R}}(K)))$ (see also \cite{S3}
for more details.

It is known (and easy to see) that $H^i_{\mathfrak{m}}(M)$ is an Artinian $R$-module for any $i$ and
a finitely generated $R$-module $M$. Then the isomorphisms follow by Theorem \ref{xx} (b) by the aid
of Matlis duality.
\end{proof}

In the particular case of a Gorenstein ring it follows that $K(\hat{R}) \cong \hat{R}$. So Corollary
\ref{yy} provides the classically known Local Duality Theorem for a Gorenstein ring. In the following we shall consider
the case of an arbitrary cohomologically complete intersection $I$ in a Gorenstein ring.

\begin{corollary} \label{zz} Let $I \subset R$ denote a cohomologically complete intersection
in a Gorenstein ring $(R,\mathfrak{m})$. Then there is the isomorphism $D_I(R) \cong
\varprojlim H^d_{\mathfrak{m}}(R/I^{\alpha})$, where $d = \dim R/I$. That is, there are natural
isomorphisms
\[
\Hom_R(H^{c-i}_I(X), E_R(k)) \simeq \Ext^i_R(X, \varprojlim H^d_{\mathfrak{m}}(R/I^{\alpha}))
\]
for any $R$-module $X$ and all $i\in \mathbb{Z}$.
\end{corollary}

\begin{proof} By the definition of local cohomology there are the following isomorphisms $H^i_I(R) \cong
\varinjlim \Ext^i_R(R/I^{\alpha}, R)$ for all $i \in \mathbb{Z}$. By the duality we get the isomorphisms
\[
\Hom_R(H^c_I(R), E_R(k)) \cong \Hom_R(\varinjlim \Ext_R^c(R/I^{\alpha},R), E_R(k)) \cong
\varprojlim H^d_{\mathfrak{m}}(R/I^{\alpha}).
\]
Note that the Hom-functor in the first place transforms a direct limit into an inverse limit.
\end{proof}

Note that $\varprojlim H^i_{\mathfrak{m}}(R/I^{\alpha})$ was studied in \cite{S2} under the name
formal local cohomology. See also \cite{S2} for more details. It is a little bit surprising that
the $d$-th formal local cohomology occurs as the duality module for the duality of cohomologically complete
intersections in a Gorenstein ring.

Now we consider the particular case of a one dimensional cohomologically complete intersection in a Gorenstein ring.

\begin{corollary} \label{11} Let $I \subset R$ denote a one dimensional cohomologically complete
intersection in a Gorenstein ring $R$ with $n = \dim R$.  Let $x \in \mathfrak{m}$ be a parameter of $R/I$ and
let $X$ denote an arbitrary $R$-module.
Then for all $i \in \mathbb{Z}$ there are natural isomorphisms
\[
\Hom_R(H^{c-i}_I(X),E_R(k)) \cong \Ext_R^i(X, D),
\]
where $D$ denotes the cokernel of the natural homomorphism $\hat{R}^I \to \hat{R_x}^I$.
\end{corollary}

\begin{proof} The proof is an obvious consequence of Corollary \ref{zz} by the aid of the results from section 1.
\end{proof}

Another interpretation of the duality module $D$ in Corollary \ref{11} can be done as the
cokernel of the natural map $\hat{R}^I \to \oplus_{i=0}^s \widehat{R_{\mathfrak{p_i}}}$ as done in
Corollary 1.2.

\section{Cohomologically complete intersections: A generalization to modules}
\label{sect_gen_cpis_to_mod}
In this section let $I$ be an ideal of a local ring $(R,\mathfrak m)$. Let  $M$ denote  a finitely generated  $R$-module.
Let $E_R^{\cdot}(M)$ denote a minimal injective resolution of the $R$-module $M$. The cohomology of the
complex $\Gamma_I(E_R^{\cdot}(M))$ is by definition the local cohomology $H^i_I(M), i \in \mathbb{N}$. Suppose that $c = \grade(I,M)$. Then $\Gamma_I(E_R^{i}(M)) = 0$ for all $i < c$. Therefore
$H^c_I(M) = \Ker(\Gamma_I(E_R^{c}(M)) \to \Gamma_I(E_R^{c+1}(M))$ and there is an embedding
$H^c_I(M)[-c] \to \Gamma_I(E_R^{\cdot}(M))$ of complexes.

\begin{definition} \label{trc}
The cokernel of the embedding $H^c_I(M)[-c] \to \Gamma_I(E_R^{\cdot}(M))$ is defined by $C^{\cdot}_M(I)$, the
truncation complex of $M$ with respect to $I$. So there is a short exact sequence
\[
0 \to H^c_I(M)[-c] \to \Gamma_I(E_R^{\cdot}(M)) \to C^{\cdot}_M(I) \to 0
\]
of complexes of $R$-modules. In particular $H^i(C^{\cdot}_M(I)) = 0$ for all $i \leq c$ and
$H^i(C^{\cdot}_M(I))) \cong H^i_I(M)$ for all $i >c$.
\end{definition}

Note that the definition of the truncation complex was used in the case of $M= R$ a Gorenstein ring
in \cite{hellus_schenzel}. This construction is used in order to obtain certain natural homomorphisms.

\begin{lemma} \label{trcl}
Let $M$ denote a finitely generated $R$-module with $c = \grade(I,M)$. Then there are natural homomorphisms
\[
H^{i-c}_{\mathfrak{m}}(H^c_I(M)) \to H^i_{\mathfrak{m}}(M)
\]
for all $i \in \mathbb{N}$.  These are isomorphisms for all $i \in \mathbb{Z}$ if and only if
$H^i_{\mathfrak{m}}(C^{\cdot}_M(I)) = 0$ for all $i \in \mathbb{Z}$.
\end{lemma}

\begin{proof} Take the short exact sequence of the truncation complex (cf. \ref{trc}) and apply the derived functor $\Rgam (\cdot).$ In the derived category this provides a short exact sequence of complexes
\[
0 \to \Rgam (H^c_I(M))[-c] \to \Rgam (\gam(E^{\cdot}_R(M))) \to \Rgam (C^{\cdot}_M(I)) \to 0.
\]
Since $\gam(E^{\cdot}_R(M))$ is a complex of injective $R$-modules we might use $\Gamma_{\mathfrak m}(\Gamma_I(E^{\cdot}_R(M)))$ as a representative of $\Rgam (\gam(E^{\cdot}_R(M))).$ But now there is an equality for the composite of section functors $\Gamma_{\mathfrak m}(\Gamma_I(\cdot)) = \Gamma_{\mathfrak m}(\cdot).$ Therefore $\Gamma_{\mathfrak m}(E^{\cdot}_R(M))$ is a representative of $\Rgam (\gam(E^{\cdot}_R(M)))$ in the derived category.

First of all it provides the natural homomorphisms of the statement.
Then the the long exact cohomology sequence provides that these maps are isomorphisms if and only if
$H^i_{\mathfrak{m}}(C^{\cdot}_M(I)) = 0$ for all $i \in \mathbb{Z}$.
\end{proof}

\begin{definition} \label{ccim} The finitely generated $R$-module $M$ is called cohomologically complete intersection
with respect to $I$ in case there is an integer $c \in \mathbb{N}$ such that $H^i_I(M) = 0$ for all $i \not=c$.
Cleary $c = \grade (I,M)$.
\end{definition}

This notion extends those of a cohomologically complete intersection $I \subset R$ in a Gorenstein ring $R$
as it was studied in \cite{hellus_schenzel}.

It is our intention now to generalize part of \cite[Theorem 5.1]{hellus_schenzel} to the situation of a module $M$
and an ideal $I \subset R$ satisfying the requirements of Definition \ref{ccim}. See also \cite{M} for similar results.

\begin{theorem} \label{th_gvococciwrtM}
Let $(R,\mathfrak m)$ be a local ring, let $M$ be a finitely generated $R$-module,
$I$ an ideal of $R$. Let $c:=\grade(I,M)$. Then the following conditions are
equivalent:
\begin{itemize}
	\item[(i)] $H^i_I(M)=0$ for all $i \not= c$.
	\item[(ii)] The natural map
	\[
	H^i_{\mathfrak{p}R_{\mathfrak{p}}}(H^c_{IR_{\mathfrak{p}}}(M_{\mathfrak{p}}))
	\to H^{i+c}_{\mathfrak pR_{\mathfrak p}}(M_{\mathfrak p})
	\]
	is an isomorphism for all $\mathfrak{p} \in V(I) \cap \Supp M$ and all $i \in \mathbb{Z}$.
\end{itemize}
\end{theorem}

\begin{proof} We begin with the proof of the implication $(i) \Longrightarrow (ii)$. By the
assumption it follows easily that $c = \grade (IR_{\mathfrak{p}}, M_{\mathfrak{p}})$ for all
$\mathfrak{p} \in V(I) \cap \Supp M$. That is we might reduce the proof to the case of the
maximal ideal. By the assumption in (i) it follows (see Definition \ref{trc}) that $C^{\cdot}_M(I)$
is an exact bounded complex. Therefore $H^i_{\mathfrak{m}}(C^{\cdot}_M(I)) = 0$ for all $i \in \mathbb{Z}$.
So the claim follows by virtue of Lemma \ref{trcl}.

For the proof of $(ii) \Longrightarrow (i)$ we proceed by an induction on $\dim V(I)\cap \Supp M =: t$.  In the
case of $t = 0$, i.e. $V(I)\cap \Supp M = \{\mathfrak{m}\}$, it follows that $\Rgam (C^{\cdot}_M(I)) \cong
C^{\cdot}_M(I)$. Then the claim is true by virtue of Lemma \ref{trcl}. Now suppose that $t > 0$ and by
induction hypothesis the statement holds for all smaller dimensions. Then it follows that
$\Supp H^i_I(M) \subseteq \{\mathfrak{m}\}$ for all $i \not= c$. By the definition of $C^{\cdot}_M(I)$ we get
that $\Supp H^i(C^{\cdot}_M(I)) \subseteq \{\mathfrak{m}\}$. Therefore it follows that $H^i_{\mathfrak{m}} (C^{\cdot}_M(I))
\cong H^i(C^{\cdot}_M(I))$ for all $i \in \mathbb{Z}$. By the assumption in (i) for $\mathfrak{p} = \mathfrak{m}$
it implies (see Lemma \ref{trcl})  that
\[
H^i_{\mathfrak{m}}(C^{\cdot}_M(I)) \cong H^i(C^{\cdot}_M(I)) = H^i_I(M) = 0
\]
for all $i \not= c.$ This completes the proof.
\end{proof}

We remark that Theorem \ref{th_gvococciwrtM} works without the hypothesis ''$R$ is Gorenstein''. In the paper
\cite{hellus_schenzel} the authors considered only the case of a Gorenstein ring.

\section{A note on direct and inverse limits}

In the proof of \cite[Lemma 1.2(a)]{hellus_schenzel} it is claimed that $\Ext$ of a direct limit in the first variable is the projective limit of the corresponding $\Ext$'s. In general, this is not true: E.~g. it is well-known and not very difficult to see that $\Ext^1_{\mathbb Z}(\mathbb Q, \mathbb Z)$ is non-zero (it is actually uncountable), while $\mathbb Q$ can be written as a direct limit of copies of $\mathbb Z$'s and each copy of $\Ext^1_{\mathbb Z}(\mathbb Z, \mathbb Z)$ is of course zero. We explain how this problem can be overcome (literally all results from \cite{hellus_schenzel} are valid -- apart from lemma 1.2 (a) ).

The general result is the following:

\begin{theorem} \label{limit} (\cite[Lemma 2.6]{S4}) Let $\{M_{\alpha}\}$ be a direct system of $R$-modules. Let $N$ denote an arbitrary $R$-module.
Then there is a short exact sequence
\[
0 \to \varprojlim{}^1 \Ext^{i-1}_R(M_{\alpha}, N) \to \Ext_R^i(\varinjlim M_{\alpha}, N) \to \varprojlim \Ext_R^i(M_{\alpha},N)
\to 0
\]
for all $i \in \mathbb{Z}$. In particular, $\Hom_R(\varinjlim M_{\alpha}, N) \cong \varprojlim \Hom_R(M_{\alpha}, N)$.
\end{theorem}

The previous Lemma \ref{limit} gives the corrected version of \cite[Lemma 1.2 (a)]{hellus_schenzel}. In the following we shall
explain how to derive the other results of \cite[Lemma 1.2]{hellus_schenzel}.

%Part (c) from \cite[Lemma 1.2]{hellus_schenzel} is correct: Because of Matlis duality it suffices to prove its second claim, namely that $\Ext^i_R(H^j_I(R),R)=0$ for all $i<j$: To compute such an $\Ext$-module we use a minimal injective resolution $0\to R\to E^\bullet $ of $R$ (we know how $E^\bullet$ looks sincs $R$ is Gorenstein); now $\Hom_R(H^j_I(R),E_R(R/\mathfrak p))=\Hom_{R_\mathfrak p}(H^j_{IR_\mathfrak p}(R_\mathfrak p),E_R(R/\mathfrak p))$ (extension of coefficients, note that $E_R(R/\mathfrak p)=E_{R_\mathfrak p}(k(\mathfrak p))$;) is zero for each prime ideal $\mathfrak p$ whose height is $<j$.
%
%In the prove of \cite[Corollary 2.9]{hellus_schenzel} there is a reference to \cite[lemma 1.2(b)]{hellus_schenzel}: However 2.9 can be easily deduced from the minimal injective resolution $0\to H^c_I(R)\to \Gamma_I(E^c)\to \Gamma_I(E^{c+1})\to \ldots $ (notation like above) of $H^c_I(R)$; note that we know what indecomposable injective modules occur in the complex $E^\bullet $ since $R$ is Gorenstein.
%
%In the proof of $(iii)\iff (iv)$ of \cite[Theorem 3.1]{hellus_schenzel} there is a reference to \cite[lemma 1.2(b)]{hellus_schenzel}: But this equivalence $(iii)\iff (iv)$ follows from the fact that one direction (the one which is needed here) of local duality works for arbitrary (not-necessarily finitely generated) modules: This was proved in \cite[6.4.1]{hellushabil} with $I:=\mathfrak m$ and $h:=\dim R$ (such that $D(H^h_I(R))=R$), but we can also prove it using our ideas and language from section \ref{sect_on_dual_cci}:

\begin{lemma}
\label{rem_ald}Let $(R,\mathfrak m,k)$ be an $n$-dimensional local Gorenstein ring. For each $R$-module $X$ there are canonical isomorphisms
\[ \Ext^{n-i}_R(X,\hat R)\simeq \Hom_R(H^i_\mathfrak m(X),E)\]
for all $i\in \mathbb Z$. Here $\hat R$ denotes the completion of $R$.
\end{lemma}

\begin{proof} The proof follows immediately from Corollary \ref{zz} or from \cite[Theorem 6.4.1]{hellushabil}. We put $I:=\mathfrak m$,. Recall that  $D_I(R)=\hat R$.
\end{proof}

Having established this general version of Local Duality, it is easy to produce the  statement of \cite[lemma 1.2 (b)]
{hellus_schenzel}.

\begin{corollary} \label{4}
\label{coro_peter}Let $I$ be a proper ideal of height $c$ in a $n$-dimensional local Gorenstein ring $(R,\mathfrak m)$. There are canonical isomorphisms
\[ \Ext^{n-i}_R(H^j_I(R),\hat R)\simeq \Hom_R(H^i_\mathfrak m(H^j_I(R)),E)\simeq \varprojlim \Ext^{n-i}_R(\Ext^j_R(R/I^\alpha,R),\hat R)\]
for all $i,j\in \mathbb Z$.
\end{corollary}

\begin{proof}
The first of the isomorphisms is a consequence of Lemma \ref{rem_ald} applied to $H^j_I(R)$. Lemma \ref{rem_ald} applied to $\Ext^j_R(R/I^\alpha,R)$ provides a family of isomorphisms
\[ \Hom_R(H^i_\mathfrak m(\Ext^j_R(R/I^\alpha,R)),E)\simeq \Ext^{n-i}_R(\Ext^j_R(R/I^\alpha,R),\hat R), \text{for all }\alpha \in \mathbb N,\]
 which are compatible with the inverse systems induced by the natural surjections. So, it induces an isomorphism
\[ \varprojlim \Hom_R(H^i_\mathfrak m(\Ext^j_R(R/I^\alpha,R)),E)\simeq \varprojlim \Ext^{n-i}_R(\Ext^j_R(R/I^\alpha,R),\hat R),\]
for all $i$ and $j$. Since the inverse limit commutes with the direct limit under $\Hom$ in the first place (see Theorem \ref{limit}) it induces an isomorphism
\[ \varprojlim \Hom_R(H^i_\mathfrak m(\Ext^j_R(R/I^\alpha,R)),E)\simeq \Hom_R(\varinjlim H^i_\mathfrak m(\Ext^j_R(R/I^\alpha,R)),E).\]
This finally completes the proof since $H^j_I(R)\cong \varinjlim \Ext^j_R(R/I^\alpha,R)$ and because local cohomology commutes with direct limits.
\end{proof}

With these results in mind the proof of \cite[Lemma 1.2 (c)]{hellus_schenzel} follows the same line of arguments as
in the original paper. In the proof of \cite[Corollary 2.9]{hellus_schenzel} there is a reference to \cite[Lemma 1.2(b)]{hellus_schenzel}: However 2.9 can be easily deduced from the minimal injective resolution $0\to H^c_I(R)\to \Gamma_I(E^c)\to \Gamma_I(E^{c+1})\to \ldots $ (where $0\to R\to E^\bullet $ is a minimal injective resolution of $R$) of $H^c_I(R)$; note that we know what indecomposable injective modules occur in the complex $E^\bullet $ since $R$ is Gorenstein.

In the proof of (iii)$\iff $(iv) of \cite[Theorem 3.1]{hellus_schenzel} there is a reference to \cite[lemma 1.2(b)]{hellus_schenzel}: But this equivalence (iii)$\iff $(iv) follows from Lemma \ref{rem_ald}.

\begin{remark} \label{5} Combining the statements in Theorem \ref{limit} with those of Corollary \ref{4} it follows
that
\[
\varprojlim{}^1 \Ext_R^i(\Ext_R^j(R/I^{\alpha}, R),\hat{R}) = 0
\]
for all $i,j \in \mathbb{Z}$.
\end{remark}


\begin{thebibliography}{1111}

\bibitem{B} {\sc M. Brodmann:}  Asymptotic stability of ${\rm Ass}(M/I\sp{n}M)$, Proc. Amer. Math. Soc. {\bf 74} (1979), no.~1, 16--18.

\bibitem{CS} {\sc F. W. Call, R. Y. Sharp:} A short proof of the local Lichtenbaum-Hartshorne theorem on the vanishing of local cohomology. 
Bull. Lond. Math. Soc. {\bf 18} (1986), 261-264 . 

\bibitem{EX} {\sc E. Enochs,  J. Xu:}  On invariants dual to the Bass numbers, Proc. Amer. Math. Soc. {\bf 125} (1997), no.~4, 951--960.

\bibitem{E} {\sc E. Enochs:}  Flat covers and flat cotorsion modules, Proc. Amer. Math. Soc. {\bf 92} (1984), no.~2, 179--184.

\bibitem{hellushabil} {\sc M. Hellus:} Local Cohomology and Matlis Duality, Ha\-bi\-li\-ta\-tions\-schrift, Leipzig,
2006, available from http://www.mathematik.uni-regensburg.de/Hellus/HabilitationsschriftOhneDeckblatt.pdf

\bibitem{hellus_schenzel} {\sc M. Hellus, P. Schenzel:} On cohomologically complete intersections, J. Algebra {\bf 320} (2008), no.~10, 3733--3748.

\bibitem{hellus_stueckrad} {\sc M. Hellus,  J. St\"uckrad:}  Matlis duals of top local cohomology modules, Proc. Amer. Math. Soc. {\bf 136} (2008), no.~2, 489--498.

\bibitem{I} {\sc T. Ishikawa:} On injective modules and flat modules, J. Math. Soc. Japan {\bf 17} (1965), 291--296.

\bibitem{M} {\sc W. Mahmood:} On endomorphism rings of local cohomology modules, in preparation.

\bibitem{matsumura} {\sc H. Matsumura:}  {\it Commutative ring theory}, translated from the Japanese by M. Reid, Cambridge Studies in Advanced Mathematics, 8, Cambridge Univ. Press, Cambridge, 1986.

\bibitem{S3} {\sc P. Schenzel:}  On The Use of Local Cohomology in Algebra and Geometry. In: Six Lectures in
Commutative Algebra, Proceed. Summer School on Commutative Algebra at Centre de
Recerca Matem\`{a}tica, (Ed.: J. Elias, J. M. Giral, R. M. Mir\'{o}-Roig, S. Zarzuela), Progr. Math. 166, pp. 241-292, Birkh\"auser, 1998.


\bibitem{S} {\sc P. Schenzel:}  Proregular sequences, local cohomology, and completion, Math. Scand. {\bf 92} (2003), no.~2, 161--180.

\bibitem{S2} {\sc P. Schenzel:}  On formal local cohomology and connectedness, J. Algebra {\bf 315} (2007), no.~2, 894--923.

\bibitem{S4} {\sc  P. Schenzel:}  On the structure of the endomorphism ring of a certain local cohomology module, J. Algebra
{\bf 344} (2011), 229-245.

\bibitem{ZZ} {\sc M. R. Zargar,  H. Zakeri:}  On injective and Gorenstein injective dimensions of local cohomology
modules, \texttt{arXiv:1204.2394}.

\end{thebibliography}
\end{document}